\documentclass[12pt]{amsart}

\usepackage[ngerman,english]{babel}

\usepackage[utf8]{inputenc}

\usepackage[T1]{fontenc}
\usepackage{ae,aecompl} 

\usepackage{hyperref}

\usepackage{a4wide}
\usepackage{amsmath}
\usepackage{amssymb}
\usepackage{amsthm}
\usepackage{latexsym}

\usepackage{enumerate} 

\usepackage{mathrsfs}

\newcommand{\setEnglish}{
\selectlanguage{english}
\def\Chapter{Chapter}
\def\Section{Section}
\def\Theorem{Theorem}
\def\Proposition{Proposition}
\def\Lemma{Lemma}
\def\Corollary{Corollary}
\def\Definition{Definition}
\def\Remark{Remark}
\def\Example{Example}
\def\Examples{Examples}
\def\Consequence{Consequence}
\def\Fact{Fact}
\def\Notation{Notation}
\def\Conjecture{Conjecture}
\def\Exercise{Exercise}
\def\AExercise{Additionalexercise}
\def\Project{Project}
\def\Problem{Problem}
\def\Observation{Observation}
}

\newtheorem{theorem}{\Theorem}
\newtheorem{proposition}[theorem]{\Proposition}
\newtheorem{lemma}[theorem]{\Lemma}
\newtheorem{corollary}[theorem]{\Corollary}

\theoremstyle{definition}

\newtheorem{definition}[theorem]{\Definition}
\newtheorem*{remark}{\Remark}

\numberwithin{theorem}{section}

\newcommand{\N}{\ensuremath{\mathbb{N}}}    % Set of natural numbers - N
\newcommand{\Z}{\ensuremath{\mathbb{Z}}}    % Set of integers - Z
\newcommand{\Q}{\ensuremath{\mathbb{Q}}}    % Set of rational numbers - Q
\newcommand{\C}{\ensuremath{\mathbb{C}}}    % Set of complex numbers - C
\renewcommand{\H}{\mathbb{H}}

\newcommand{\SLZ}{\SL_2(\Z)}

\renewcommand{\phi}{\varphi}
\renewcommand{\epsilon}{\varepsilon}

\DeclareMathOperator{\SL}{SL}

\DeclareMathOperator{\Gal}{Gal}

\newcommand{\set}[1]{\left\{ #1 \right\}}
\newcommand{\br}[1]{\left( #1 \right)}
\newcommand{\abs}[1]{\left|#1\right|}    % absolute value

\renewcommand{\b}{b_{m,k,N}(l,v)}
\newcommand{\GN}{\Gamma_0(N)}

\title{Class invariants for certain non-holomorphic modular functions}

\author{Joschka J. Braun}
\address{Technische Universität Darmstadt, D-64289 Darmstadt, Germany}
\email{joschka.braun@gmail.com}
\author{Johannes J. Buck}
\address{Technische Universität Darmstadt, D-64289 Darmstadt, Germany}
\email{buck@stud.tu-darmstadt.de}
\author{Johannes Girsch}
\address{Fakultät für Mathematik, Universität Wien, A-1090 Vienna, Austria}
\email{johannes.girsch@live.de}
\date{\today}

\begin{document}
\setEnglish

\begin{abstract}
%!TEX root = main.tex

Inspired by previous work of Bruinier-Ono and Mertens-Rolen, we study class polynomials for non-holomorphic modular functions arising from modular forms of negative weight.
%In particular, we give general conditions for the irreducibility of class polynomials. 
In particular, we give general conditions for the irreducibility of class polynomials and obtain a general theorem to check when functions constructed in a special way are class invariants.
%This allows us to easily generate infinitely many new class invariants.
\end{abstract}

\maketitle
%\tableofcontents

\section{Introduction and Statement of the Main Results}
%!TEX root = main.tex

Let us consider Klein's $j$-invariant, $j(\tau)$, which is defined by

\[
j(\tau):=\frac{\left(1+240 \sum_{n=1}^{\infty} \sum_{d|n} d^3 q^n\right)^3}{q\prod_{n=1}^{\infty}(1-q^n)^{24}}=q^{-1}+744+196884q+\cdots,
\]
where $q:=e^{2\pi i\tau}$ and $\tau:=u+iv\in\mathbb{H}:=\{x+iy\in\mathbb{C} : y>0\}$ with $u,v\in\mathbb{R}$. The function $j(\tau)$ is a modular function and its values at CM points (quadratic imaginary points in $\mathbb{H}$) are called \emph{singular moduli}. These distinguished numbers play a central role in explicit class field theory and the classical theory of complex multiplication. In particular, they are algebraic, and they generate \emph{Hilbert class fields} of imaginary quadratic fields. For a general survey of this theory, see e.g. \cite{borel1966seminar}.

Throughout, let $D\equiv 0,1 \pmod 4$ be a negative integer and let $\mathcal{Q}_D$ be the set of $\SLZ$-reduced positive definite integral binary quadratic forms, i.e., those forms

$$
Q(x,y)=ax^2+bxy+cy^2,
$$
with $-a<b\leq a<c\text{ or } 0\leq b\leq a=c$ and discriminant $D:=b^2-4ac$. Each positive definite quadratic form of discriminant $D$ is uniquely $\SLZ$-equivalent to exactly one form in $\mathcal{Q}_D$. The associated complex point $\tau_Q$ of a given quadratic form $Q$ is the unique point in the upper half plane which satisfies the equation $Q(\tau_Q, 1)=0$. We call the associated value $j(\tau_Q)$ the singular modulus corresponding to $\tau_Q$.

Now we introduce the following polynomial:

\[
H_{D, P}(x):=\prod_{Q\in \mathcal{Q}_D} (x-P(\tau)),
\]

for any (possibly non-holomorphic) modular function $P(\tau)$.

For $P(\tau)=j(\tau)$, $H_{D, P}(x)$ is called the \emph{Hilbert class polynomial}. Celebrated results give that it is irreducible and generates the Hilbert class field of $\mathbb{Q}(\sqrt{D})$ for $D$ fundamental. For more general discriminants, it generates ring class fields (see \cite{borel1966seminar}).

Analogously, we set
\[
\widehat{H}_{D, P}(x):= \prod_{Q\in\mathcal P_D}(x-P(\tau_Q)),
\]
where $\mathcal P_D$ denotes the set of quadratic forms in $\mathcal{Q}_D$ with $\gcd(a,b,c)=1$. The quadratic forms in $\mathcal P_D$ are called \emph{primitive}.

One can easily see as in Lemma 3.7 of \cite{bruinier2013class} the following relation between $H_{D, P}$ and $\widehat{H}_{D, P}$:

\[
H_{D, P}(x)=\prod_{\substack{a>0\\a^2|D}}\epsilon(a)^{h\left(\frac{D}{a^2}\right)}\widehat{H}_{\frac{D}{a^2}, P}(\epsilon(a)x)
\]
where $h(d)$ is the class number of discriminant $d$, $\epsilon(a)=1$ if $f\equiv \pm 1 \pmod{12}$ and $\epsilon(a)=-1$ otherwise.

In pathbreaking work, Bruinier and Ono connected these polynomials to partitions, when $P$ was an explicit non-holomorphic form of level $6$, yielding a formula for the number of partitions as a finite sum of algebraic numbers. They, together with Sutherland, conjectured that the associated polynomials always generate ring class fields and are irreducible (\cite{bruinier2013algebraic}, p. 20), which was shown by Mertens and Rolen in \cite{mertens2015class}. Many others have studied properties of these non-holomorphic singular moduli (see \cite{alfes2014formulas,bruinier2013class,griffin2013ramanujan,griffinproperties,larson2011integrality}). Here, we consider the general problem of constructing class invariants (i.e., modular forms whose CM-values generate Hilbert class fields) from non-holomorphic modular forms of a special type.

In particular, we show the following, where $B$ is defined in \eqref{eq:b} and,
% (which will be defined in \eqref{eq:b})
% is the sum of $B_0$, $B_1$, $B_2$, $B_3$ and $B_4$, which are defined in \eqref{eq:b0}, \eqref{eq:b1}, \eqref{eq:b2}, \eqref{eq:b3} and \eqref{eq:b4} and,
% in the proof of Theorem \ref{main-theorem} and,
in the setting of Theorem \ref{main-theorem}, is a sum which depends on $m$ and $k$, where $P_F$ is an iterated non-holomorphic derivative of $F$ of weight $0$, defined in \ref{pf}, and where the polynomial $\widehat{H}_{D,F}(x)$ is defined in \ref{hdf}. To prove Theorem \ref{main-theorem}, we proceed as in \cite{mertens2015class}.

\begin{theorem} \label{main-theorem}
Suppose $k \in \mathbb{N} \setminus \{0\}$ and $F$ is a weakly holomorphic modular form of weight $-2k$ on $\SLZ$ with principal part 
$$\sum_{n=1}^{m}a_n q^{-n}$$
and rational Fourier coefficients. If $D$ is a fundamental discriminant satisfying
\[
\sqrt{-D}>\frac{2}{\pi}\log{\left(\frac{B\sum_{n=1}^{m}|a_n|}{|a_m||\sum_{r=0}^k\binom{k}{r}\frac{(-2k+r)_{k-r}m^r}{(2\sqrt{-D}\pi )^{k-r}}|}e^{-\sqrt{-D}\pi(m-1/2)}\right)},
\]
then the polynomial $\widehat{H}_{D,F}(x)$ is irreducible over $\mathbb{Q}$.
\end{theorem}

\begin{remark}
The bounds used in this paper can be easily adapted for higher levels. Therefore, once one explicitly knows the higher level Heegner points, a similar analysis could be done to generate class invariants. For example, in level $6$ for quadratic forms of discriminant $-24n+1$ the Heegner points have been determined by Dewar and Murty, see \cite{dewar2013derivation}. In particular, Dewar and Murty's results were used in the paper of Mertens and Rolen (\cite{mertens2015class}).
\end{remark}

Note that it is easy to show that the singular moduli in Theorem \ref{main-theorem} lie in the appropriate Hilbert class field (see Lemma 4.4 in \cite{bruinier2013algebraic}). In particular, we have the following.

\begin{corollary}
If $D\ll 0$ is fundamental, then $\widehat{H}_{D,F}(x)$ is irreducible and the splitting field of this polynomial is isomorphic to the Hilbert class field of $\Q (\sqrt{D})$.
\end{corollary}

These bounds can also be simplified to give the following.

\begin{corollary} \label{main-coro}
Assume the hypothesis in Theorem \ref{main-theorem}. If there is a constant $c>1$ such that
\[
\sqrt{-D}> \max \left(\frac{k}{m\pi\left(\sqrt[k]{\frac{2c-1}{c}}-1\right)}, \frac{2}{\pi}\log\left(615cm^{1+k}\left(m+\frac{k}{\sqrt{3}\pi}\right)^k\frac{\sum_{n=1}^{m}|a_n|}{|a_m|}\right) \right),
\]
then $\widehat{H}_{D,F}(x)$ is irreducible. In particular, for all fundamental discriminants $D'$ with $|D'|\geq|D|$  the polynomials $\widehat{H}_{D',F}(x)$ are irreducible.
\end{corollary}

\iffalse %Hat sich erübrigt. In der aktuellen Form ist er nicht mehr inhaltsreich.
\begin{remark}
These bounds depend on the parameter $c$. In Section \ref{example}, we give an explicit example and find an appropriate $c$.% and discuss how to make an optimal choice of $c$ in a special case.
\end{remark}
\fi

The paper is organized as following. In Section \ref{preliminaries} we review relevant background information including Masser's formula and a convenient form of Shimura reciprocity due to Schertz and we recall the Maass-Poincaré series. The proof of Theorem \ref{main-theorem} and Corollary \ref{main-coro} is subject of Section \ref{proof}. In Section \ref{example} we apply Corollary 1.3 in the specific example of $F=E_{10}/\Delta$.

\section*{Acknowledgements}
%!TEX root = main.tex

This project was written at the Cologne Young Researchers in Number Theory Program 2015. The authors wish to thank the organizer Larry Rolen, for his generous support and guidance throughout the project. Furthermore, they want to recognize the generous support and contribution from the DFG Grant D-72133-G-403-151001011 of Larry Rolen, which was funded under the Institutional Strategy of the University of Cologne within the German Excellence Initiative. They would also like to thank Michael Griffin for insightful conversations. Further, the authors wish to express their gratitude to Claudia Alfes, Kathrin Bringmann, and Michael H. Mertens for their valuable comments and correspondence. Lastly, they would like to thank the anonymous referees for their helpful comments.

\section{Preliminaries} \label{preliminaries}
%!TEX root = main.tex

\subsection{Differential operators}

Let $F$ be a weakly holomorphic form of weight $-2k$.
We apply the Maass raising operator, defined for $l\in\mathbb{N}$ and $\tau =u+iv$ by
\[
R_l:=\frac{1}{2\pi i}\frac{\partial}{\partial\tau}-\frac{l}{4\pi v},
\]
$k$ times to $F$ to get a non-holomorphic modular function of weight $0$
\begin{equation}
\label{pf}
P_F(\tau):=R_{-2}\circ\cdots \circ R_{-2k+2}\circ R_{-2k}(F)(\tau).
\end{equation}

The Maass raising operator maps a (not 	necessarily holomorphic) modular form of weight $k$ to a (possibly) non-holomorphic modular form of weight $k+2$. It is an easy way to raise a modular form of negative weight to a non-holomorphic modular function (see Bump \cite{bump1998automorphic}, note there is used another definition).

As in \cite{zagier2008elliptic}, one may easily check that

\begin{equation}
R_{-2}\circ\cdots \circ R_{-2k+2}\circ R_{-2k}(F)(\tau) = \sum_{r=0}^{k} \binom{k}{r} \frac{(k+r)_{k-r}}{(4\pi v)^{k-r}}\mathcal{D}^r(F)(\tau)
\label{Zag}
\end{equation}
where $(a)_m:=a(a+1)\cdots(a+m-1)$ is the Pochhammer symbol and $\mathcal{D}:=q\frac{d}{dq}$.

We then set
\begin{equation}
\widehat{H}_{D,F}(x):=\prod_{Q\in \mathcal{P}_{D}} (x-P_F(\tau_Q)),
\label{hdf}
\end{equation}
where $\tau_Q$ is the CM-point of the quadratic form $Q$. As noted in \cite{bruinier2013algebraic} on page 2 and \cite{bruinier2013class} on page 3, $\widehat{H}_{D,F}(x)\in\Q[x]$.
%!TEX root = main.tex

\subsection{Masser's formula}

If one applies the Maass raising operator $k$ times to a weakly holomorphic modular form of weight $-2k$, one gets a non-holomorphic modular function.
%We have to interpolate the non-holomorphic modular function by a holomorphic modular function at CM-points of certain quadratic forms to be able to apply later Schertz's theorem (which is a special case of the general theory of Shimura reciprocity).
%Interpolation means that we use another function with the same values at these CM-points.
We would like to find a holomorphic modular function which has the same values at CM-points of certain quadratic forms as our non-holomorphic modular function to be able to apply Schertz's theorem (which is a special case of the general theory of Shimura reciprocity). The non-holomorphic modular function $P_F$ we obtain is an \emph{almost holomorphic modular form}. The ring of almost holomorphic modular forms is the ring of functions which transform as a modular form but instead of being holomorphic they are polynomials in $1/v$ with coefficients which are holomorphic. It is well-known that this ring is generated by $E_2(\tau)-3/\pi v, E_4(\tau), E_6(\tau)$ where $E_2$, $E_4$ and $E_6$ are the normalized Eisenstein series of weight $2$, $4$ and $6$. The modular form $E_2(\tau)-3/\pi v$ is also known as $E_2^*(\tau)$. For a survey of almost holomorphic modular functions, see e.g. Zagier \cite{zagier2008elliptic}.

Hence, we can express the almost holomorphic modular function $P_F$ as a polynomial in these three generators. Because $E_4$ and $E_6$ are holomorphic, the only function we have to consider is $E_2^*$.
%Masser states a useful formula to interpolate $E_2^*$ at quadratic imaginary points in $\mathbb{H}$ by a (meromorphic) modular form, see Appendix I in Masser \cite{masser1975elliptic}.
Masser states a useful formula to find a (meromorphic) modular form which has the same values as $E_2^*$ at quadratic imaginary points in $\mathbb{H}$, see Appendix I in Masser \cite{masser1975elliptic}. For this purpose, let $\mathcal V$ be a system of representatives of $\SLZ\backslash\Gamma_{-D}$ for a discriminant $D<0$ with associated positive definite integral binary quadratic form $Q$ and corresponding CM point $\tau_Q$ and let $\Gamma_{-D}$ be the set of all primitive integral $2\times2$-matrices of determinant $-D$. Then we introduce the \emph{modular polynomial} by
\[
\Phi_{-D}(j(\tau),y):=\prod_{M\in\mathcal V} (y-j(M\tau)).
\]

By expanding this polynomial, we can define the numbers $\beta_{\mu, \nu}(\tau_Q)$ by
\[
\Phi_{-D}(x,y)=\sum_{\mu, \nu}\beta_{\mu, \nu}(\tau_Q)(x-j(\tau_Q))^\mu(y-j(\tau_Q))^\nu.
\]
Now we can state the relation which Masser gives.
\begin{proposition}[\cite{masser1975elliptic}]

If $Q$ is a positive definite integral binary quadratic form of discriminant $D<0$ and $\tau_Q$ is the associated CM point, then
\[
E_2^*(\tau_Q)=\left(4\frac{E_6}{E_4}+3\frac{E_4^2}{E_6}+6j\gamma\frac{E_6}{E_4}\right)(\tau_Q),
\]
where $\gamma=\frac{\beta_{4,0}-\beta_{3,1}-\beta_{1,3}-\beta_{0,4}}{\beta_{0,1}}$ if $D$ is special and $\gamma=\frac{\beta_{2,0}-\beta_{1,1}-\beta_{0,2}}{\beta_{0,1}}$ otherwise. $D$ is called special if it is of the form $D=-3d^2$.
\end{proposition}

Thus we get the following.

\begin{lemma} \label{prop-masser}
For every discriminant $D<0$, there exists a (meromorphic) modular function $M_{D,F}$ such that
\[
P_F(\tau_Q)=M_{D,F}(\tau_Q)
\]
for all positive definite integral binary quadratic forms $Q$ of discriminant $D$.
\end{lemma}

This lemma reduces us to studying classical modular functions, where work of Schertz applies.
%!TEX root = main.tex

\subsection{Maass-Poincaré series}

In the proof of Theorem \ref{main-theorem} we will have need of careful estimates for the coefficients of weakly holomorphic modular forms. This is conveniently provided by the theory of Maass-Poincaré series due to Niebur \cite{niebur1973class} and Fay \cite{fay1977fourier}, as further developed by many others (see e.g. \cite{bringmann2008coefficients,bringmann2007arithmetic}).

For $v>0$, $k\in\Z$ and $s\in\C$ we define
$$\mathcal{M}_{s,k}(v):=v^{-\tfrac{k}{2}} M_{-\tfrac{k}{2},s-\tfrac{1}{2}}\left(v\right)$$
where $M_{\nu,\mu}$ denotes the usual $M$-Whittaker function (see e.g. \cite{gradshteyn2000table}, p. 1014). Using this, we construct the following Poincaré series for $\GN$:
$$P_{m,s,k,N}(\tau) := \frac{1}{2\Gamma(2s)} \sum_{\gamma \in \Gamma_\infty \setminus \GN} \br{\mathcal M_{s,k}(4\pi m v) e^{-2\pi i m u}} \rvert_k \gamma.$$
Here $m\in\N$, $\tau=u+iv \in\H$, $\Re(s)>1$, $k\in -\N$, and $\Gamma_\infty:=\set{\pm \br{\begin{smallmatrix} 1&n\\0&1  \end{smallmatrix}} \mid n\in\Z}$.
%the stabilizer of the cusp $\infty$.

In the special case when $k<0$ and $s=1-\tfrac{k}{2}$, the Poincaré series $P_{m,k,N}:=P_{m,s,k,N}$ defines a harmonic Maass form of weight $k$ for $\GN$ whose principal part at the cusp $\infty$ is given by $q^{-m}$, and at all other cusps the principal part is 0 (which is not of importance for our purposes because we are only interested in level $N=1$). For more background on Poincaré series and harmonic Maass forms see \cite{bruinier2004two,ono2009unearthing}. 

In this situation, we can give explicit Fourier expansions.
In order to do so we agree on the following notation.
By $I_s$ and $J_s$ we denote the usual $I$- and $J$-Bessel functions, and $K(m,l,c)$ is the usual Kloosterman sum
$$K(m,l,c):= \sum_{d {\pmod c}^*} \exp{ \br { 2\pi i \br{ \frac{m\overline d + ld} {c}} }},$$
where $d$ runs through the residue classes modulo $c$ which are coprime to $c$ and $\overline d$ denotes the multiplicative inverse of $d$ modulo $c$.

\begin{proposition}[\cite{ono2009unearthing}, Theorem 8.4] \label{poincare-exp} 
For $m,N\in\N$, $k\in -\N$ and $\tau\in\H$ we have
$$(1-k)! \: P_{m,k,N} (\tau) = (k-1) (\Gamma (1-k, 4\pi mv) - \Gamma(1-k)) q^{-m} + \sum_{l\in\Z} \b q^l,$$
where the incomplete gamma function $\Gamma(\alpha,x)$ is defined as
$$\Gamma(\alpha,x) := \int_x^\infty e^{-t} t^{\alpha-1} dt$$
and the coefficients $\b$ are as follows.

\begin{enumerate}[1)]
\item If $l<0$, then
\begin{align*}
\b = 2\pi i^{2-k} (k-1) \Gamma(1-k,-4\pi l v) \br{-\frac{l}{m}}^{\frac{k-1}{2}}\\
\times \sum_{c=1}^\infty \frac{K(-m,l,cN)}{cN} J_{1-k} \br{\frac{4\pi \sqrt{-ml}}{cN}}.
\end{align*}
\item If $l>0$, then
\begin{align*}
\b = -2\pi i^{2-k} (1-k)! l^{\frac{k-1}{2}} m^{\frac{1-k}{2}} \sum_{c=1}^\infty \frac{K(-m,l,cN)}{cN} I_{1-k} \br{\frac{4\pi \sqrt{ml}}{cN}}.
\end{align*}
\item If $l=0$, then
\begin{align*}
\b = -(2\pi i)^{2-k} m^{1-k} \sum_{c=1}^\infty \frac{K(-m,0,cN)}{(cN)^{2-k}}.
\end{align*}
\end{enumerate}

\end{proposition}
%!TEX root = main.tex

\iffalse
\begin{definition}
Suppose we have a weakly holomporphic modular form $F$ of weight $-2k$. We then define the almost holomporphic modular form
$$P_F(\tau):=(R_{-2}\dotsc R_{-2k}F)(\tau).$$
And a polynomial
$$\hat{H}_{D,F}(x)=\prod_{Q\in P_\delta}(x-P_F(\tau_q))$$
where $P_\delta$ is the set of primitive forms in $Q_\delta$.
\end{definition}
\fi

\subsection{Work of Schertz}
In this subsection we review work of Schertz which gives a convenient description of the singular moduli of modular functions.
\begin{definition}
Let $N\in\mathbb N$ and $D=t^2d<0$ be a discriminant with $t\in\mathbb N$ and $d$ a fundamental discriminant. Moreover, let $\{Q_1,\dotsc,Q_h\}$ be a system of representatives of primitive quadratic forms modulo $\SLZ$. We call the set $\{Q_1,\dotsc,Q_h\}$ an \emph{$N$-system mod $t$} if  for each $Q_i(x,y)=a_ix^2+b_ixy+c_iy^2$ the conditions
$$\gcd(c_j,N)=1 \quad \text{and} \quad  b_j\equiv b_l\pmod{2N}, \quad 1\leq j,l\leq h$$
are satisfied.
\end{definition}
\begin{theorem}[Schertz,\cite{schertz2002weber}] \label{schertz} 
Let $g$ be a modular function for $\Gamma_0(N)$ for some $N\in\mathbb N$ whose Fourier coefficients at all cusps lie in the $N$th cyclotomic field. Suppose furthermore that $g(\tau)$ and $g(-\frac{1}{\tau})$ have rational Fourier coefficients, and let $Q(x,y)=ax^2+bxy+cy^2$ be a quadratic form with discriminant $D=t^2d$, $d$ a fundamental discriminant, with $\gcd(d,N)=1$ and $N|a$. Then, unless $g$ has a pole at $\tau_Q$, we have that $g(\tau_Q)\in\Omega_t$, where $\Omega_t$ is the ring class field of the order of conductor $t$ in $\mathbb Q(\sqrt{d})$. Moreover if $\{Q=Q_1,Q_2,\dotsc,Q_h\}$ is an $N$-system mod $t$, then
$$\{g(\tau_{Q_1}),\dotsc,g(\tau_{Q_h})\}=\{\sigma(g(\tau_{Q_1})):\sigma\in\Gal_D\}$$ 
where $\Gal_D$ is the Galois group of $\Omega_t/\mathbb Q(\sqrt{D})$.
\end{theorem}
\begin{remark}
Schertz also proves constructively that an $N$-system mod $t$ always exists.
\end{remark}

%\subsection{Quadratic forms}
%In \cite{zagier2008elliptic} it is stated that every $\SLZ$-equivalence class of quadratic forms with discriminant $D$ has a unique representative in $\mathcal{Q}_D$.
%We prove that in this set there is one representative where the absolute value of the corresponding CM-point is larger than the others. This will be an important ingredient of our proof of irreducibility of the class polynomials.

\iffalse
\begin{proposition}
For every discriminant $D$ there is exactly one equivalence class whose representantive in $\mathcal{Q}_D$ has $a=1$.
\end{proposition}
\begin{proof}
We consider two cases:
\begin{enumerate}
\item $D\equiv 0\mod 4$: The form $\left[1,0,-D/4\right]$ certainly belongs to $\mathcal{Q}_D$. Now for every representative with $a=1$ in $\mathcal{Q}_D$ we have that $b=0$. We immediately get that $c=-D/4$.
\item $D\equiv 1\mod 4$: We have $\left[1,1,(1-D)/4\right]\in \mathcal{Q}_D$. Again suppose there exists a second representative with $a=1$. Then we get $b=1$ and so $c=(1-D)/4$.
\end{enumerate}
\end{proof}
\fi

\section{Proof of Theorem \ref{main-theorem}} \label{proof}
%!TEX root = main.tex
Now we have all the needed preliminaries and begin the proof. We proceed as in \cite{mertens2015class}.
\begin{lemma}\label{estilem}
Suppose we have a weakly modular form $F$ of weight $-2k$ with principal part 
$$\sum_{n=1}^{m}a_n q^{-n}.$$
Then
$$R_{-2}\circ\dotsc\circ R_{-2k}(F)(\tau)=R_{-2}\circ\dotsc\circ R_{-2k}\left(\sum_{n=1}^{m}a_n q^{-n}\right)(\tau)+E_F(\tau)$$
where $|E_F(\tau)|$ is uniformly bounded by the constant in \eqref{eq:final}.
\end{lemma}

\begin{proof}As the space of cusp forms of weight zero for the full modular group is zero, we can write $F$ as a sum of Poincaré series, namely as
$$F=\sum_{n=1}^{m}\frac{a_n}{(1+2k)!} P_{n,-2k,1}.$$
Hence,
$$|E_F|\leq\sum_{n=1}^{m}\left|\frac{a_n}{(1+2k)!} E_{P_{n,-2k,1}}\right|.$$
We can estimate the Fourier coefficients of a given Poincaré series by using that for $v>-1/2$ and $0\leq x \leq 1$ we have the well-known bounds (see e.g., eq. (6.25) in \cite{luke1972inequalities} and eq. (8.451.5) in \cite{gradshteyn2000table})
$$I_v(x)\leq\frac{2}{\Gamma(v+1)}\left(\frac{x}{2}\right)^v$$
and for $x\geq 1$
$$I_v(x)\leq\frac{e^x}{\sqrt{2\pi x}}.$$

Using this we find for $c\geq 4\pi\sqrt{nl}$ that
$$I_{1+2k}\left(\frac{4\pi\sqrt{nl}}{c}\right)\leq\frac{2}{\Gamma(2+2k)}\left(\frac{4\pi\sqrt{nl}}{2c}\right)^{1+2k}=\frac{2^{2+2k}\pi ^{1+2k}(nl)^{1/2+k}}{c^{1+2k}(1+2k)!}$$
and for $c< 4\pi\sqrt{nl}$ that
$$I_{1+2k}\left(\frac{4\pi\sqrt{nl}}{c}\right)\leq\frac{\sqrt{c}e^{\frac{4\pi\sqrt{nl}}{c}}}{2\pi\sqrt{2}\sqrt[4]{nl}}.$$

Using the trivial bound $|K(-n,l,c)|\leq c$ one obtains
\begin{align*}
|b_{n, -2k, 1}(l, v)| &= 2\pi (1+2k)! l^{\frac{-2k-1}{2}} n^{\frac{1+2k}{2}} \left|\sum_{c=1}^{\infty} \frac{K(-n, l, c)}{c} I_{1+2k}\left(\frac{4\pi\sqrt{nl}}{c}\right)\right| \\
										&\leq 2\pi (1+2k)! l^{-1/2-k} n^{1/2+k} \sum_{c=1}^\infty \left|I_{1+2k}\left(\frac{4\pi\sqrt{nl}}{c}\right)\right|.\\
\end{align*}		
Applying the estimates for the Bessel functions	stated above we get		
\begin{align*}
&\sum_{c=1}^\infty \left|I_{1+2k}\left(\frac{4\pi\sqrt{nl}}{c}\right)\right|\\
=&\sum_{c<4\pi\sqrt{nl}}\left|I_{1+2k}\left(\frac{4\pi\sqrt{nl}}{c}\right)\right|+\sum_{c\geq 4\pi\sqrt{nl}}\left|I_{1+2k}\left(\frac{4\pi\sqrt{nl}}{c}\right)\right|\\
\leq&\frac{2^{2+2k}\pi ^{1+2k}(nl)^{1/2+k}}{(1+2k)!} \zeta(1+2k)+2\sqrt{2\pi nl}e^{4\pi\sqrt{nl}}.
\end{align*}

Finally, one obtains
$$|b_{n, -2k, 1}(l, v)|\leq 4\sqrt{2}\pi^{3/2}(1+2k)!l^{-k}n^{1+k}e^{4\pi\sqrt{nl}}+2^{3+2k}\pi^{2+2k} n^{1+2k}\zeta(1+2k).$$					

For the constant term of the Fourier-expansion we get
\begin{align*}
|b_{n, -2k, 1}(0, v)| &= \left|(2\pi)^{2+2k} n^{1+2k} \sum_{c=1}^{\infty} \frac{K(-n, 0, c)}{c^{2+2k}}\right|\\
											&\leq (2\pi)^{2+2k} n^{1+2k} \zeta(1+2k). 
\end{align*}
Continuing, by (\ref{Zag}) we obtain
\begin{align*}\left|\frac{E_{P_{n,-2k,1}}}{(1+2k)!}\right|&=\left|R_{-2}\circ\dotsc \circ R_{-2k}\left(\sum_{l=0}^{\infty}\frac{b_{n,-2k,1}(l)}{(1+2k)!}q^l\right)\right|\\
&=\left|\sum_{r=0}^k(-1)^{k-r}\binom{k}{r}\frac{(-2k+r)_{k-r}}{(4\pi v)^{k-r}}\mathcal{D}^r\left(\sum_{l=0}^{\infty}\frac{b_{n,-2k,1}(l)}{(1+2k)!}q^l\right)\right|\\
&\leq\sum_{r=0}^k\binom{k}{r}\frac{(2k)^{k-r}}{(|4\pi v|)^{k-r}}\sum_{l=0}^{\infty}l^r \left|\frac{b_{n,-2k,1}(l)}{(1+2k)!}q^l\right|.
\end{align*}
Since $\tau$ lies in the closure of the usual fundamental domain of the full modular group, i.e. $\tau\in \overline{\lbrace z\in\mathbb H \mid \abs{z}>1,\abs{\Re(z)}<\frac{1}{2}\rbrace}$ and hence $v\geq\sqrt{3}/2$, we can estimate
\begin{align*}
\left|\frac{E_{P_{n,-2k,1}}}{(1+2k)!}\right|&\leq\frac{(2\pi)^{2+2k}n^{1+2k}\zeta(1+2k)}{(1+2k)!}+\left|\sum_{r=0}^k\sum_{l=1}^{\infty}\binom{k}{r}\frac{(k)^{k-r}}{(\sqrt{3}\pi )^{k-r}}l^r \frac{b_{n,-2k,1}(l)}{(1+2k)!}q^l\right|\\
&\leq\frac{(2\pi)^{2+2k}n^{1+2k}\zeta(1+2k)}{(1+2k)!}+\sum_{l=1}^{\infty}\left|\frac{b_{n,-2k,1}(l)}{(1+2k)!}q^l\right|\sum_{r=0}^k\binom{k}{r}\frac{(k)^{k-r}l^r}{(\sqrt{3}\pi )^{k-r}}\\
&=\frac{(2\pi)^{2+2k}n^{1+2k}\zeta(1+2k)}{(1+2k)!}+\sum_{l=1}^{\infty}\left|\frac{b_{n,-2k,1}(l)}{(1+2k)!}\right||q|^l\left(l+\frac{k}{\sqrt{3}\pi}\right)^k.\\
\end{align*}
Now we use the estimates for the coefficients of Poincaré series which were established above and get
$$\left|\frac{E_{P_{n,-2k,1}}}{(1+2k)!}\right|\leq\frac{(2\pi)^{2+2k}n^{1+2k}\zeta(1+2k)}{(1+2k)!}+A_1+A_2,$$
where 
$$A_1:=4\pi\sqrt{2\pi}n^{1+k}\sum_{l=1}^{\infty}\left(1+\frac{k}{\sqrt{3}\pi l}\right)^k e^{4\pi\sqrt{nl}-\sqrt{3}\pi l}$$
and 
$$A_2:=\frac{2^{3+2k}\pi^{2+2k}n^{1+2k}\zeta(1+2k)}{(2k+1)!}\sum_{l=1}^{\infty}\left(l+\frac{k}{\sqrt{3}\pi}\right)^ke^{-\sqrt{3}\pi l}.$$

Now we estimate the first sum by
$$A_1\leq 4\pi\sqrt{2\pi}n^{1+k}\sum_{l=1}^{\infty}\left(1+\frac{k}{\sqrt{3}\pi}\right)^k e^{4\pi\sqrt{nl}-\sqrt{3}\pi l}.$$

Since for any real-valued function $f$ satisfying $f'(x)>0$ we have the inequality
$$\sum_{k=M}^N f(k)\leq\int_{M}^N f(x)dx+f(N)$$ 
and a similar result also holds if $f'(x)<0$ this last sum can be estimated against an integral
$$\sum_{l=1}^{\infty}e^{4\pi\sqrt{nl}-\sqrt{3}\pi l}\leq\int_{1}^{\infty}e^{4\pi\sqrt{n\mu}-\sqrt{3}\pi \mu}d\mu+e^{\frac{4\pi n}{\sqrt{3}}}.$$
The integral can be evaluated explicitly, which easily yields the following bound 
$$\int_{1}^{\infty}e^{4\pi\sqrt{n\mu}-\sqrt{3}\pi\mu}d\mu\leq\frac{4\sqrt{n}e^{\frac{4\pi n}{\sqrt{3}}}}{3^{3/4}}+\frac{e^{4\pi\sqrt{n}-\sqrt{3}\pi}}{\sqrt{3}\pi}.$$
Using these estimates, one obtains
\begin{align*}
A_1&\leq 4\pi\sqrt{2\pi}n^{1+k}\left(1+\frac{k}{\sqrt{3}\pi}\right)^k\left(\frac{4\sqrt{n}e^{\frac{4\pi n}{\sqrt{3}}}}{3^{3/4}}+\frac{e^{4\pi\sqrt{n}-\sqrt{3}\pi}}{\sqrt{3}\pi}+e^{\frac{4\pi n}{\sqrt{3}}}\right)\\
&\leq 12\pi\sqrt{2\pi}n^{3/2+k}\left(1+\frac{k}{\sqrt{3}\pi}\right)^ke^{\frac{4\pi n}{\sqrt{3}}}.
\end{align*}

For the second sum we first note that for $x\geq 0$ we have the inequality
$$\left(x+\frac{k}{\sqrt{3}\pi}\right)^ke^{-\sqrt{3}\pi x/2}\leq\left(\frac{2k}{\sqrt{3}\pi}\right)^ke^{-k/2}.$$
We can therefore estimate
$$A_2\leq\frac{2^{3+2k}\pi^{2+2k}n^{1+2k}\zeta(1+2k)}{(2k+1)!}\left(\frac{2k}{\sqrt{3}\pi}\right)^ke^{-k/2}\frac{e^{-\sqrt{3}\pi/2}}{1-e^{-\sqrt{3}\pi/2}}.$$
Thus,
\begin{equation}
\label{eq:final}|E_F|\leq\sum_{n=1}^{m}\left|a_n \frac{E_{P_{n,-2k,1}}}{(1+2k)!}\right|\leq \left|\frac{E_{P_{m,-2k,1}}}{(1+2k)!}\right|\sum_{n=1}^{m}|a_n|\leq\sum_{n=1}^{m}|a_n|\left(B_0+B_1+B_2\right),\end{equation}
where
\begin{align}
\label{eq:b0}
B_0&:=\frac{2^{3+2k}\pi^{2+2k}m^{1+2k}\zeta(1+2k)}{(1+2k)!}\leq 1064m^{1+2k},\\
\label{eq:b1}
B_1&:=24\pi\sqrt{2\pi}m^{3/2+k}\left(1+\frac{k}{\sqrt{3}\pi}\right)^ke^{\frac{4\pi m}{\sqrt{3}}}\leq 189m^{3/2+k}\left(1+\frac{k}{\sqrt{3}\pi}\right)^ke^{\frac{4\pi m}{\sqrt{3}}}
\end{align}
and
\begin{align}
\label{eq:b2}
B_2&:=\frac{2^{4+2k}\pi^{2+2k}m^{1+2k}\zeta(1+2k)}{(2k+1)!}\left(\frac{2k}{\sqrt{3}\pi}\right)^ke^{-k/2}\frac{e^{-\sqrt{3}\pi/2}}{1-e^{-\sqrt{3}\pi/2}}\leq 245m^{1+2k}\left(\frac{k}{\sqrt{3}\pi}\right)^k,
\end{align}
which completes the proof.

\end{proof}

We are now in position to prove Theorem \ref{main-theorem}.

\iffalse
\begin{theorem}
Suppose we have a weakly modular form $F$ of weight $-2k$ with principal part 
$$\sum_{n=1}^{m}a_n q^{-n}$$
where $a_n\in\mathbb Q$ and $m>1$. Then for large enough fundamental discriminant $-D$, the polynomials $\hat{H}_{D,F}(x)$ are irreducible over $\mathbb Q$. 
\end{theorem}
\fi

\begin{proof}[Proof of Theorem \ref{main-theorem}]
By Theorem \ref{schertz} and Lemma \ref{prop-masser}, we immediateley see that the Galois group of $\Omega_t/\mathbb Q(\sqrt{D})$ acts transitively on the roots and therefore the polynomial must be a power of an irreducible polynomial. We now prove that for large enough $D$ there cannot be multiple roots. In \cite{zagier2008elliptic} it is stated that every $\SLZ$-equivalence class of quadratic forms with discriminant $D$ has a unique representative in $\mathcal{Q}_D$. It is easy to prove that in this set there is exactly one representative $[a,b,c]$, where $a=1$. Suppose that $P(\tau_{Q_1})=P(\tau_{Q_2})$ where $\tau_{Q_1}$ is the CM-point of the quadratic form with $a=1$. We will see that for large enough discriminant $-D$, the value at this point will be larger than the rest, so we cannot have multiple roots.
We set $e^{2\pi i\tau_{Q_j}}=:q_j$ for $j=1,2$.  One then obtains that
$$|q_1|=e^{\pi\sqrt{-D}} \quad\text{and}\quad |q_2|\leq e^{\pi\sqrt{-D}/2}.$$

Using Lemma \ref{estilem} we have
$$\left|R_{-2}\circ\dotsc\circ R_{-2k}\left(\sum_{n=1}^{m}a_n q^{-n}\right)(\tau_{Q_1})-R_{-2}\circ\dotsc\circ R_{-2k}\left(\sum_{n=1}^{m}a_n q^{-n}\right)(\tau_{Q_2})\right|\leq2\sum_{n=1}^{m}|a_n| E_n.$$
Hence one obtains
\begin{multline*}\left|R_{-2}\circ\dotsc\circ R_{-2k}\left(a_m q^{-m}\right)(\tau_{Q_1})\right|\leq2\sum_{n=1}^{m}|a_n| E_n\\
+\left|\sum_{r=0}^k(-1)^{k-r}\binom{k}{r}\frac{(-2k+r)_{k-r}}{(4\pi y)^{k-r}}\mathcal{D}^r\left(\sum_{n=1}^{m-1}a_n q^{-n}\right)(\tau_{Q_1})\right|\\+\left|\sum_{r=0}^k(-1)^{k-r}\binom{k}{r}\frac{(-2k+r)_{k-r}}{(4\pi y)^{k-r}}\mathcal{D}^r\left(\sum_{n=1}^{m}a_n q^{-n}\right)(\tau_{Q_2})\right|.
\end{multline*}

Using some estimates in the spirit of Lemma \ref{estilem} we get that
$$\left|\sum_{r=0}^k(-1)^{k-r}\binom{k}{r}\frac{(-2k+r)_{k-r}}{(4\pi v)^{k-r}}\sum_{n=1}^{m-1}(-n)^r a_n q_1^{-n}\right|\leq\left(m+\frac{k}{\sqrt{3}\pi}\right)^ke^{\sqrt{-D}\pi(m-1)}\sum_{n=1}^{m}|a_n|$$
and analogously
$$\left|\sum_{r=0}^k(-1)^{k-r}\binom{k}{r}\frac{(-2k+r)_{k-r}}{(4\pi v)^{k-r}}\sum_{n=1}^{m}(-n)^r a_n q_2^{-n}\right|\leq\left(m+\frac{k}{\sqrt{3}\pi}\right)^ke^{\sqrt{-D}\pi m/2}\sum_{n=1}^{m}|a_n|.$$

We also have that 
\begin{align*}
|R_{-2}\circ\dotsc\circ R_{-2k}\left(a_m q^{-m}\right)(\tau_{Q_1})|&=|a_m|e^{-\sqrt{-D}\pi m}\left|\sum_{r=0}^k(-1)^{k-r}\binom{k}{r}\frac{(-2k+r)_{k-r}(-m)^r}{(2\sqrt{-D}\pi )^{k-r}}\right|.
\end{align*}

Hence, if the polynomial is reducible we must have that 
$$e^{\sqrt{-D}\pi m}|a_m|\left|\sum_{r=0}^k\binom{k}{r}\frac{(-2k+r)_{k-r}(m)^r}{(2\sqrt{-D}\pi )^{k-r}}\right|<
\sum_{n=1}^{m} B|a_n|,$$
where $B_0$, $B_1$ and $B_2$ are defined in \eqref{eq:b0}, \eqref{eq:b1}, \eqref{eq:b2} and
\begin{align*}
%\label{eq:b3}
B_3&:=\left(m+\frac{k}{\sqrt{3}\pi}\right)^ke^{\sqrt{-D}\pi(m-1)},\\
%\label{eq:b4}
B_4&:=\left(m+\frac{k}{\sqrt{3}\pi}\right)^ke^{\sqrt{-D}\pi m/2}
\end{align*}
and
\begin{align}
\label{eq:b}
B:=B_0+B_1+B_2+B_3+B_4.
\end{align}

Equivalently, we have
$$
\sqrt{-D}<\frac{2}{\pi}\log{\left(\frac{B \sum_{n=1}^{m}|a_n|}{|a_m|\left|\sum_{r=0}^k\binom{k}{r}\frac{(-2k+r)_{k-r}(m)^r}{(2\sqrt{-D}\pi )^{k-r}}\right|} e^{-\sqrt{-D}\pi(m-1/2)}\right)}.$$

Note that as a function of $-D$ the right hand side of the last inequality stays bounded so for large enough $-D$ this inequality cannot be satisfied and hence the polynomials must be irreducible.
\end{proof}

\iffalse
\begin{corollary}
For every $a>0$ if $$\sqrt{-D}>\frac{k}{m\pi\left(\sqrt[k]{\frac{2a-1}{a}}-1\right)}$$
and
$$\sqrt{-D}>\frac{1}{\pi}\log\left(3000am^{1+k}(m+\frac{k}{\sqrt{3}\pi})^k\frac{\sum_{n=1}^{m}|a_n|}{|a_m|}\right)$$
then the polynomials are irreducible.
\end{corollary}
\fi

\begin{proof}[Proof of Corollary \ref{main-coro}]
Since all negative fundamental discriminants $D$ with $|D|<15$ have class number one, these polynomials are automatically irreducible. Therefore we assume that $-D\geq 15$.
We have
$$\left|\sum_{r=0}^k\binom{k}{r}\frac{(-2k+r)_{k-r}(m)^r}{(2\sqrt{-D}\pi )^{k-r}}\right|>m^k-\sum_{r=0}^{k-1}\binom{k}{r}\frac{k^{k-r}(m)^r}{(\sqrt{-D}\pi )^{k-r}}=2m^k-\left(m+\frac{k}{\sqrt{-D}\pi}\right)^k.$$
A simple calculation gives that for $c>1$ and 
$$\sqrt{-D}>\frac{k}{m\pi\left(\sqrt[k]{\frac{2c-1}{c}}-1\right)}$$
the following inequality holds:
$$2m^k-\left(m+\frac{k}{\sqrt{-D}\pi}\right)^k>\frac{m^k}{c}.$$ 
Since for $-D\geq 15$ 
$$m \mapsto e^{\frac{4\pi m}{\sqrt{3}}-\sqrt{-D}\pi(m-1/2)}$$
is a decreasing function we can estimate 
$$B_1e^{-\sqrt{-D}\pi(m-1/2)}\leq 610m^{1+2k}\left(m+\frac{k}{\sqrt{3}\pi}\right)^k.$$
Similarly we obtain 
\begin{align*}
B_0e^{-\sqrt{-D}\pi(m-1/2)}&\leq 2.5m^{1+2k}\left(m+\frac{k}{\sqrt{3}\pi}\right)^k,\\
B_2e^{-\sqrt{-D}\pi(m-1/2)}&\leq 0.6m^{1+2k}\left(m+\frac{k}{\sqrt{3}\pi}\right)^k,\\
B_3e^{-\sqrt{-D}\pi(m-1/2)}&\leq 0.0023m^{1+2k}\left(m+\frac{k}{\sqrt{3}\pi}\right)^k
\end{align*}
and
$$B_4e^{-\sqrt{-D}\pi(m-1/2)}\leq m^{1+2k}\left(m+\frac{k}{\sqrt{3}\pi}\right)^k.$$
Therefore the polynomial is irreducible if
$$\sqrt{-D}>\frac{k}{m\pi\left(\sqrt[k]{\frac{2c-1}{c}}-1\right)}$$
and 
$$\sqrt{-D}>\frac{2}{\pi}\log\left(615c m^{1+k}\left(m+\frac{k}{\sqrt{3}\pi}\right)^k\frac{\sum_{n=1}^{m}|a_n|}{|a_m|}\right).$$
\end{proof}

\section{Example} \label{example}
%!TEX root = main.tex

We consider $F:=E_{10}/\Delta$, a weakly holomorphic modular form of weight $-2$. 
Zagier previously considered this example in his foundational paper on singular moduli \cite{zagier2000traces}.
Since $E_{10}$ does not vanish at infinity and $\Delta$ is a cusp form with a simple root at infinity, $F$ has a simple pole at infinity. Thus we have
$$F=q^{-1}+O(1).$$
Hence, in the setting of Corollary \ref{main-coro} we have $m=k=b_1=1$. Therefore, the two inequalities reduce to
$$ \sqrt{-D} > \max \br{ \frac{c}{\pi(c-1)} , \frac{2}{\pi} \log\br{615c  \br{1+\frac{1}{\sqrt 3 \pi}} } }.$$
For $c=1.5$ we get
$$\sqrt{-D} > 4.45366,$$
which proves that $\widehat{H}_{D,F}(x)$ is irreducible for $D$ being a fundamental discriminant $D\le -20$. The fundamental discriminants $0 > D >-20$ except $D=-15$ have class number $h(D)=1$, so $\widehat{H}_{D,F}(x)$ is of degree 1 and therefore also irreducible. We cannot solve the $D=-15$ case with Corollary \ref{main-coro} since
$$\frac{2}{\pi} \log\br{615c  \br{1+\frac{1}{\sqrt 3 \pi}} } > \sqrt{15}$$
for all $c>1$.
However, for $D=-15$ we numerically compute in SAGE (see \cite{sage})
$$\widehat{H}_{D,F}(x)=x^2 + 176625x+ 9890505,$$
which is easily seen to be irreducible (in fact, its roots generate the field $\Q(\sqrt{5})$).

\bibliographystyle{plain}
\bibliography{references}

\end{document}